\documentclass[12pt, letterpaper]{article}
\usepackage[utf8]{inputenc}
\usepackage[english]{babel}
\usepackage{amsmath}
\usepackage{amsfonts}
\usepackage{amssymb}
\usepackage{amsthm}
\usepackage{bbm}
\usepackage{color}
\usepackage{fullpage}
\usepackage{hyperref}
\usepackage{graphicx}
\usepackage[affil-it]{authblk}
\newtheorem{defi}{Definition}
\newtheorem{teo}{Theorem}
\newtheorem{lem}{Lemma}

\newtheorem{prop}{Proposition}
\newtheorem{remark}{Remark}
\newtheorem{cor}{Corollary}

\newcommand{\st}{\leq_{\mathrm{st}}}

\newcommand{\stdom}{\leq_{\sigma,k}}
\newcommand{\stdompar}[1]{\leq_{#1}}
\newcommand{\1}[1]{\mathbbm{1}_{#1}}
\newcommand{\R}{\mathbb{R}}
\newcommand{\esp}[1]{\mathbb{E}\left[#1\right]}
\newcommand{\var}[1]{\text{Var}\left(#1\right)}

\numberwithin{equation}{section}

\makeatletter
\def\@maketitle{%
  \newpage
  \null
  \vskip 2em%
  \begin{center}%
  \let \footnote \thanks
    {\Large\bfseries \@title \par}%
    \vskip 1.5em%
    {\normalsize
      \lineskip .5em%
      \begin{tabular}[t]{c}%
        \@author
      \end{tabular}\par}%
    \vskip 1em%
    {\normalsize \@date}%
  \end{center}%
  \par
  \vskip 1.5em}
\makeatother

\title{On the links between Stein transforms and concentration inequalities for dependent random variables}

\author{Santiago Arenas-Velilla
  \thanks{Electronic address: \texttt{santiago.arenas@cimat.mx}}}
\affil{Centro de Investigación en Matemáticas, Guanajuato, MEXICO}

\author{Emilien Joly
  \thanks{Electronic address: \texttt{emilien.joly@cimat.mx}; Corresponding author.\\
  The research of this author was partially funded by CONACYT, Ciencia de Frontera project 1043167.}}
\affil{Centro de Investigación en Matemáticas, Guanajuato, MEXICO}

\begin{document}

\maketitle

\begin{abstract}
    In this paper, we explore some links between transforms derived by Stein's method and concentration inequalities. In particular, we show that the stochastic domination of the zero bias transform of a random variable is equivalent to sub-Gaussian concentration. For this purpose a new stochastic order is considered. In a second time, we study the case of functions of slightly dependent light-tailed random variables. We are able to recover the famous McDiarmid type of concentration inequality  for functions with the bounded difference property. Additionally, we obtain new concentration bounds when we authorize a light dependence between the random variables. Finally, we give a analogous result for another type of Stein's transform, the so-called size bias transform.
\end{abstract}

\section{Introduction}

The so-called Stein's method has proven to be a very powerful tool to characterize the distribution of random variables (or random processes) in many different contexts. In particular, Stein's method is very well suited for proving the asymptotic normality of a sequence of random variables under very weak assumptions. The advantages of this technique are twofold. The first one is the possibility to quantify the distance in measure of some distribution to the normal distribution. Stein original paper \cite{MR0402873} states a characterization of normality of a random variable through a class of equations that this random variable verifies. More concretely, for a real random variable $Z$ and a positive constant $\sigma$, we have that  $Z \sim N(0, \sigma^2)$ if and only if
\begin{equation}\label{steinequation}
    \mathbb{E}[Zf(Z)] = \sigma^2 \mathbb{E}[f'(Z)]
\end{equation}
for all absolutely continuous functions $f$ for which the above expectations exists. The quantitative results that has been derived from these ideas rely on the unique solution $f_h$ of the differential equation
\begin{equation}\label{steinequation-functional}
    y'(x)-xy(x)=h(x)-\mathbb{E}h,
\end{equation}
where $h$ is any fixed real valued measurable function with finite first moment $\mathbb{E}h=\mathbb{E}h(N)$ for the standard Gaussian variable $N$. The special case $h_t(x)=\1{x\le t}$ is of special importance since it allows to prove the characterization \eqref{steinequation}. The simple form of Equation \eqref{steinequation-functional} permit to write explicit solutions and then consider the regularity of those solutions. This path has been specially fruitful in quantifying the distance between the distribution of $Z$ and the normal distributed variable (with the same variance $\sigma^2$). Such ideas lead naturally to get Berry-Essen type of bounds (see for example \cite{chen2001non}).
A particular case of interest is when the random variable $Z$ is itself function of a class of random variables $X_1,\dots,X_n$. In that case, one can study the effect of each random variable $X_i$ on $Z$ through a directional version of equation \eqref{steinequation-functional}.

The second advantage of Stein's method is the possibility to manage weak dependence in the random variables family $X_1,\dots,X_n$. We refer to \cite{chen2004normal} for examples of those ideas to prove asymptotic normality.
In this paper we consider another approach that is somewhat parallel to the characterization of Equation \eqref{steinequation-functional}. Since in general a real valued random variable does not satisfies Equation \eqref{steinequation}, we can force the equality by introducing a new distribution called the \textit{zero bias transform} of $X$ firstly defined in \cite{MR1484792}.

\begin{defi}\label{defi_zerobias}
Let $X$ be a random variable with mean zero and finite variance $\sigma^2$. The random variable $X^*$ has the zero bias distribution with respect to $X$ if
\begin{equation}\label{zerobias}
\mathbb{E}[Xf(X)] = \sigma^2 \mathbb{E}[f'(X^{*})]
\end{equation}
for any absolutely continuous function $f$ such that the expectations exists.
\end{defi}

This definition gives naturally a mapping that associates to a random variable $X$ another random variable noted $X^*$. Note that this construction is only distributional. In other terms, no dependence structure between $X$ and $X^*$ is forced by this definition. We will benefit from that in the proof of Theorem \ref{teo:subgaussianfunctional}.
One can legitimately ask for a justification of the existence of such a construction. Such considerations and more are recalled (for the sake of completeness) in the following proposition that is inspired by Lemma 2.1 from \cite{MR1484792} and Proposition 2.1 from \cite{chen2011normal}).

\begin{prop}\label{prop_Zerobiasproperties}
Let $X$ be a random variable with mean zero and finite variance $\sigma^2$ and let $X^*$ having the zero bias distribution with respect to $X$, then
\begin{enumerate}
    \item The Gaussian distribution with mean zero and variance $\sigma^2$ is the unique fixed point of the equation \eqref{zerobias}.
    \item The zero bias distribution is unimodal about zero and preserves symmetry.
    \item The distribution of $X^*$ is absolutely continuous with density given by
\begin{equation}\label{zerobiasdensity}
    f_{X^*}(t) = \frac{1}{\sigma^2}\mathbb{E}[X \mathbb{I}_{X > t}] = - \frac{1}{\sigma^2}\mathbb{E}[X \mathbb{I}_{X \leq t}].
\end{equation}
    \item The distribution function of $X^*$ is given by
\begin{equation}\label{zerobiasdistribution}
    \mathbb{P}(X^* \leq x) = \frac{1}{\sigma^2}\mathbb{E}[X(X-x) \mathbb{I}_{X \leq x}].
\end{equation}
    \item The support of $X^*$ is the closed convex hull of the support of $X$, and $X^*$ is bounded whenever $X$ is bounded.
    \item $(aX)^* = a X^*$ for any constant $a \neq 0$.
\end{enumerate}
\end{prop}

Similarly to Equation \eqref{steinequation}, another characterization has been developed and put in good use, the so-called Poisson characterization. It says that a positive random variable $Z$ has the Poisson distribution of parameter $\lambda>0$ if and only if
\begin{equation}\label{poisson_characterization}
    \mathbb{E}[Zf(Z)] = \lambda \mathbb{E}[f(Z+1)],
\end{equation}
for all function $f$ for which the left hand side exists. This second characterization leads naturally to a second mapping given in the following definition.
\begin{defi}\label{def:sizebias}
Let $X$ be a positive random variable with mean $\mu$ finite. The random variable $X^s$ has the size bias distribution with respect to $X$ if for all $f$ such that $\mathbb{E}[Xf(X)] < \infty$, we have
$$\mathbb{E}[Xf(X)] = \mu \mathbb{E}[f(X^s)].$$
\end{defi}
In the case that $X$ is a discrete random variable with probability mass function $f_X$ or where $X$ have density $f_X$, we obtain that
\begin{equation}\label{sizebiasdensity}
    f_{X^s}(x) = \frac{x}{\mu}f_X(x).
\end{equation}

\paragraph{On asymptotic normality}

It is possible to obtain a Berry-Essen type theorems for a random variable $X$ in the case that a bounded size bias or a zero bias coupling exist. The bounds applies for the Kolmogorov distance which is given by
$$d_{Kol}(X,Y) := \sup_{z \in \mathbb{R}} |F_X(z) - F_Y(z)|,$$
where $F_X$ and $F_Y$ are the distribution functions of the random variables $X$ and $Y$.

In the case of the zero bias transform, the next result  \cite[Theorem 5.1]{chen2011normal}, only requires the construction of a bounded zero bias coupling.

\begin{teo}\label{teo:KolmogorovZerobias}
Let $X$ be a random variable with zero mean and variance 1, and suppose that there exist $X^*$, having the zero bias distribution of $X$, defined in the same space satisfying $|X-X^*| \leq \delta$, then
\begin{equation} \label{kolmogorov_zerobias}
   d_{Kol}(X,N) \leq c \delta,
\end{equation}
where $c = 1 + 1/\sqrt{2 \pi} + \sqrt{2\pi}/ 4  \leq 2.03$ and $N$ is a standard Gaussian random variable
\end{teo}

For the size bias transform, besides the existence of a bounded size bias coupling, the calculation of terms involving a conditional expectation its also necessary \cite[Theorem 2.1]{10.1214/09-AAP634}.
\begin{teo}\label{teo:KolmogorovSizebias}
Let $X$ be a non negative random variable with finite mean $\mu$ and finite, positive variance $\sigma^2$, and suppose $X^s$, have the size bias distribution of $X$, may be coupled to $X$ so that $|X^s-X| \leq A$, for some $A$. Then with $W = (X- \mu)/ \sigma$,
\begin{equation}\label{Kolmogorov_sizebias}
d_{Kol}(W,N) \leq  \frac{\mu}{6 \sigma^2} \left(\sqrt{\frac{11 A^2}{\sigma} + \frac{5\sigma^2}{\mu}D}+ \frac{2A}{\sqrt{\sigma}} \right)^2,
\end{equation}
where $N$ is a standard Gaussian random variable, and $D$ is given by
\begin{equation}\label{D}
    D = \mathbb{E} \left| \mathbb{E}\left[ 1- \frac{\mu }{\sigma}(X^s - X)|X \right] \right|.
\end{equation}
\end{teo}
In some applications it is enough to use that $D \leq \mu \Psi / \sigma^2$, where  $\Psi$ is given by
\begin{equation}\label{Psi}
\Psi = \sqrt{\mathrm{Var}(\mathbb{E}[X^s-X | X])},
\end{equation}
in order to obtain \cite[Theorem 5.6]{chen2011normal}
\begin{equation}\label{Kolmogorov_sizebiasPsi}
d_{Kol}(W,N) \leq   \frac{6 \mu A^2}{\sigma^3} +\frac{2 \mu \Psi}{\sigma^2}.
\end{equation}

The last theorem can be used to prove normal approximation in discrete structures such that the number of crossings in graphs and trees \cite{arenasvelilla2022convergence,arenas2022, paguyo2021convergence} and the Betti numbers in random simplicial complexes\cite{MR3485340}. The previous results give us information about the convergences rates in the normal approximation for the random variable $X$, but this don't give us information about concentration. This is one of the objectives of this work.

\paragraph{On stochastic domination}
In order to exploit the result of Theorem \ref{teo:KolmogorovSizebias}, it is helpful to get domination results on quantities that involve the random variable $X^*-X$. The task is then to be able to find couplings of the two variables $X$ and $X^*$ such that their difference remains small. One way to obtain such controls is to rely on stochastic domination. We say that $X$ is \textit{stochastically dominated} by $Y$ if for all increasing and bounded function $f$,
\begin{equation*}
    \esp{f(X)}\le \esp{f(Y)}.
\end{equation*}
We denote that fact $X\le_{st}Y$. In particular, the stochastic domination implies that the cumulative distribution functions are bounded one by the other : $\forall t\in \R$, $F_X(t)\le F_Y(t)$.
We also introduce a slightly different stochastic order. A random variable $X$ is \textit{stochastically dominated in the convex order} by $Y$ if for any convex function $g$ such that $\esp{g(X)}<\infty$ and $\esp{g(Y)}<\infty$, we have that
\begin{equation*}
    \esp{g(X)}\le \esp{g(Y)}.
\end{equation*}
We will denote $X\le_{cx} Y$.
This last definition implies in particular that $\esp{X}=\esp{Y}$ and that the variance of $X$ is dominated by the variance of $Y$. This notion of stochastic order is often used in information theory where most of the entropy functions are in fact convex.
The best advantage of this formulation is its flexibility since no specific coupling is imposed on the pair of random variables. Our following results will be formulated in the context of stochastic domination whereas in all the results given in the literature, the controls on $X^*-X$ are given by almost sure domination on a particular coupling. One can see that the two notions can be linked thanks to Strassen's theorem (see \cite{strassen1965existence}) that we recall here for completeness.

\begin{teo}[Strassen stochastic order]
Let $X$ and $Y$ be such that $X\le_{st} Y$ then there exists a coupling $\pi$ which is a probability measure on $\R^2$ with marginals having the same laws as $X$ and $Y$ and such that
\begin{equation*}
    M=\{(x,y): x\le y\}
\end{equation*}
is such that $\pi(M)=1$.
\end{teo}

We refer to \cite{lindvall1999strassen} for a comprehensive proof of this theorem.
The following result gives an easy but useful characterization of the stochastic domination. (see for example \cite{MR2265633}).

\begin{lem}
Let $X$ and $Y$ be two random variables, then the followings are equivalents
\begin{enumerate}
    \item $X \st Y$.
    \item $\mathbb{P}(X \geq x) \leq \mathbb{P}(Y \geq x), \qquad \forall x \in \mathbb{R}$.
\end{enumerate}
\end{lem}

\section{Subgaussianity and zero bias control}

 \subsection{Real case}

In this section, we explore some equivalence between properties on the Stein transform of a random variable and how light tailed the random variable actually is. For this purpose, we first introduce the notion of sub-Gaussian random variables. A real valued random variable $X$ is said to be \textit{sub-Gaussian} of constant $k^2$ (and denoted by $\mathcal{G}(k^2)$) if for all $\lambda \in \mathbb{R}$, we have that
\begin{equation}
\label{eq:twosidessubgaussian}
    \esp{e^{\lambda X}} \le e^{\frac{\lambda^2k^2}{2}}.
\end{equation}

If we only assume that Equation \eqref{eq:twosidessubgaussian} holds for positive (resp. negative) $\lambda$, we say that the random variable $X$ has a \textit{sub-Gaussian right tail}, denoted $\mathcal{G}_+(k^2)$ (resp. \textit{sub-Gaussian left tail}, denoted $\mathcal{G}_-(k^2)$).
Obviously, a random variable that is both sub-Gaussian in its left and right tail is sub-Gaussian.
In the following proposition, we sum up different sufficient conditions for a real valued random variable to be sub-Gaussian.

\begin{prop}\label{prop:subgaussiancondition}
    Let $X$ be a centered real valued random variable. We denote by $M(\lambda)=\esp{e^{\lambda X}}$ its moment generating function.
    If for all $\lambda \ge 0$, we have that $M'(\lambda)\le k^2\lambda M(\lambda)$,  then $X$ is $\mathcal{G}_+(k^2)$.
\end{prop}

\begin{proof}
    Assume that a function $f$ satisfies the differential inequality $y'\le k^2 x y$ so that
    \begin{equation*}
        y'e^{-\frac{k^2x^2}{2}}-k^2xye^{-\frac{k^2x^2}{2}} = (ye^{-\frac{k^2x^2}{2}})'\le 0.
    \end{equation*}
    Then, by direct integration, we get that
    \begin{equation*}
        ye^{-\frac{k^2x^2}{2}}\le y(0)=1,
    \end{equation*}
    which proves 1.
\end{proof}

In the next definition, we introduce the notion of weighted stochastic domination that will be key in our following results.

\begin{defi}\label{def:stdom}
Let $X$ and $Y$ be two random variables and let $\sigma, k$ be two positive constants. We say that $X$ is stochastically dominated in the weighted order with constant $\sigma,k$  if $\sigma^2 \mathbb{E}f(Y) \leq k^2 \mathbb{E}f(X)$ for all increasing and positive function $f$ for which the expectations exists. We will denote this by $Y \stdom X$.
\end{defi}
Note that when the two constants $\sigma$ and $k$ coincide, the previous definition is equivalent to $Y\le_{st} X$.
In the applications, we will always take $\sigma^2=\var{X}$ and $k^2$ will be an upper bound of the variance of $Y$. In some sense, this definition is a stochastic domination weighted by the variances of the two variables.
In the case that $X^* \stdom X$, the domination $\sigma^2 \mathbb{E}f'(X^*) \leq k^2 \mathbb{E}f'(X)$ can be rewritten thanks to the equation $\sigma^2 \mathbb{E}f'(X^*) = \mathbb{E}Xf(X)$ into $\mathbb{E}Xf(X) \leq k^2 \mathbb{E}f'(X)$. This last formulation has the advantage to be define only thanks to the original random variable $X$.
In the following theorem, we give an equivalence between the weighted stochastic domination of $X^*$ by $X$ and the fact that $X$ is sub-Gaussian.

\begin{teo}\label{Teo:subGaus_equivalences}
Let $X$ be a random variable with zero mean, variance $\sigma^2$ finite and let $X^*$ his zero bias transform. Then, $X^* \stdom X$ implies that $X$ has a sub-Gaussian right tail with constant $k^2$. If $X \le_{k,\sigma} X^*$, then $X$ has a sub-Gaussian left tail with constant $k^2$. Moreover, if $X$ has density $f_X$ such that $\log f_X(x) + x^2/(2k^2)$ is concave, then $X \le_{k,\sigma} X^* \stdom X$.
\end{teo}

We point out that the first implication of Theorem \ref{Teo:subGaus_equivalences} is a slight generalization of an already known fact firstly shown in \cite[Theorem 2.1]{goldstein2014concentration}. In this result, the condition is of the form $X^*\le X+c$ for a positive constant $c$ which would be quasi equivalent to $X^* \leq_{\sigma,\sigma} X+c $.
This case can be easily obtained with the same kind of arguments as a corollary of Theorem \ref{Teo:subGaus_equivalences}. We state this small result in Corollary \ref{cor:cpositiva}.
In fact, the result in \cite{goldstein2014concentration} uses an almost sure domination that is not much stronger than stochastic domination as seen in Strassen's theorem. When the constant $c$ is equal to 0, our condition is actually weaker since we allow the two constants $\sigma$ and $k$ to be different.
For the necessary condition, we see that the hypothesis $\log f_X(x) + x^2/(2k^2)$ is concave implies the fact that $X$ is sub-Gaussian of constant $k^2$. It is actually a little stronger since it forces the density $f_X(x)$ to remain under the density of a $\mathcal{N}(0,k^2)$ for $x$ sufficiently large. This is not the case for completely general sub-Gaussian random variables.
The reader familiar with the notions of log-concavity will note that this means that the random variable $X$ is strongly log-concave (see \cite[Definition 2.8]{saumard2014log} for a rigorous definition).

\begin{proof}[Proof of Theorem \ref{Teo:subGaus_equivalences}]
Notice that for any fixed $\lambda > 0$, the function $f(x) =  e^{\lambda x}$ is increasing and convex so that the function $f'$ verifies the conditions of Definition \ref{def:stdom}. Then
$$\sigma^2 \mathbb{E}\lambda e^{\lambda X^*} \leq k^2 \mathbb{E}\lambda e^{\lambda X}.$$
If $M(\lambda)$ denotes the moment generation function of $X$, we obtain
$$M'(\lambda) = \mathbb{E}X e^{\lambda X} = \sigma^2 \mathbb{E}\lambda e^{\lambda X^*} \leq k^2 \mathbb{E}\lambda e^{\lambda X} = k^2 \lambda M(\lambda),$$
that is
$$M'(\lambda) \leq \lambda k^2 M(\lambda),  \text{ for } \lambda \geq 0,$$
which implies by Proposition \ref{prop:subgaussiancondition} that $X \in \mathcal{G}_{+}(k^2)$.

On the other hand, let $N$ be a gaussian random variable with zero mean, variance $k^2$ and density function $\phi$. Using the Proposition \ref{prop_Zerobiasproperties} we obtain for any absolutely continuous increasing and convex function $f$,
$$\sigma^2 \mathbb{E}f'(X^*) = \mathbb{E}\left[Xf(X)\right]=  \int xf(x)\frac{f_X(x)}{\phi(x)} \phi(x)dx  =  \mathbb{E}\left[ Nf(N) g(N)\right],$$
where $g$ is the quotient of the densities of $X$ and $N$. Using the zero bias characterization  for the Gaussian distribution with the function $f(x)g(x)$, and the fact that the Gaussian distribution is the fixed point of \eqref{zerobias}, we obtain
$$\mathbb{E}\left[ Nf(N)g(N)\right] = k^2 \mathbb{E}\left[f'(N)g(N)  + f(N)g'(N)\right].$$
Notice that $\mathbb{E}\left[f'(N) g(N) \right]  = \mathbb{E}f'(X) $. On the other hand, for the density of the Gaussian random variable $N$, $\phi'(x) = -x \phi(x)/k^2$, therefore

\begin{align*}
    \mathbb{E}\left[f(N)g'(N)\right] &= \mathbb{E}\left[f(N)\frac{f_X'(N)\phi(N)- f_X(N)\phi'(N)}{\phi^2(N)}\right] \\
    &= \int f(x) \left(\frac{f_X'(x)}{\phi(x)} - \frac{f_X(x)\phi'(x)}{\phi^2(x)} \right)\phi(x)dx \\
    &= \int f(x) \left(\frac{f_X'(x)}{f_X(x)} - \frac{\phi'(x)}{\phi(x)} \right)f_X(x)dx \\
    &= \int f(x) \left(\frac{f_X'(x)}{f_X(x)} + \frac{x}{k^2} \right)f_X(x)dx \\
    &= \mathbb{E}f(X)H(X).
\end{align*}
But,
\begin{equation*}
  \mathbb{E}H(X) = \int \left(\frac{f_X'(x)}{f_X(x)} + \frac{x}{k^2} \right)f_X(x)dx  = \int f'_X(x)dx + \int \frac{x}{k^2} f_X(x)dx = 0.
\end{equation*}
Also, since $f$ is increasing, the FKG inequality give us
$$\mathbb{E}f(X)h(X) \geq \mathbb{E}f(X) \mathbb{E}h(X),$$
for any increasing function $h$. Observe that the hypothesis $\log f_X(x) + x^2/(2k^2)$ be concave implies that $-H$ is increasing, then
$$\mathbb{E}f(X)(-H(X)) \geq \mathbb{E}f(X) \mathbb{E}(-H(X)) = 0,$$
that is, $\mathbb{E}f(X)H(X) \leq 0$,  which implies
$$\sigma^2 \mathbb{E}f'(X^*) =  k^2 \mathbb{E}f'(X) + k^2  \mathbb{E}\left[f(N)g'(N)\right] \leq  k^2 \mathbb{E}f'(X),$$
i.e, $X^* \stdom X$. To show the result for the left tail, we only have to consider the random variable $-X$ and see that $X\le_{k,\sigma}X^*$ is equivalent to $(-X)^*\le_{\sigma,k} -X$.
\end{proof}

The following result only study the first implication of Theorem \ref{Teo:subGaus_equivalences} when we add a little freedom in the control of $X^*$ by $X$. The cost of this change is a different regime in the tail for large values. This behavior is sometimes referred as sub-Gamma right tail.

\begin{cor}
\label{cor:cpositiva}
Let $X$ be a real-valued random variable of finite variance $\sigma^2$ and such that $X^*\le_{\sigma,k} X+c$ for a positive constants $k$ and $c$. Then, we obtain that
\begin{equation*}
    \mathbb{P}(X-\esp{X}\ge x) \le \exp \left(-\frac{x^2}{2(k^2+cx)}\right)
\end{equation*}
\end{cor}

The proof of this result is very similar to the first part of Theorem \ref{Teo:subGaus_equivalences}. The main change is that the control of the function $M$ is now of the form
\begin{equation*}
    M'(\lambda) \le k^2 \lambda e^{c\lambda}M(\lambda) \le \frac{k^2 \lambda}{1-c\lambda}M(\lambda)
\end{equation*}
when $\lambda < 1/c$. This allows to get the result. See \cite{boucheron2013concentration} for more details on sub-Gamma random variables.

In the proof of Theorem \ref{Teo:subGaus_equivalences}, we have used a change of measure to a Gaussian random variable. This particular trick allows to use Stein's equation for the Gaussian. If one wants to make a more general change of measure, a similar result can be derived for the necessary condition of Theorem \ref{Teo:subGaus_equivalences}. In this case, it is no more possible to stochastically dominate the random variable $X^*$ by $X$ but the two tails can still be compared up to a distortion term that depends on the measure we compare $X$ to. This is the purpose of the following proposition.

\begin{prop}\label{Prop:zerobiasKernel}
Let $X$ and $Y$ be two random variables with zero mean, variance $\sigma^2$  and densities $f_X$ and $f_Y$ respectively. If exist $x_0  \geq 0$ such that
\begin{enumerate}
    \item $f_X(t) /f_Y(t)$ is decreasing in $t \geq x_0$,
    \item for $t \geq x \geq x_0$,  $f_{Y^*}(t)/f_Y(t)$ is bounded by a constant $a_Y(x)$, only depending of $x$, where $f_{Y^*}$ is the density of $Y^*$ having the zero bias distribution of $Y$.
\end{enumerate} Then, for $X^*$ having the zero bias distribution of $X$,
\begin{equation}\label{eqzerobiasprobkerne}
    \mathbb{P}(X^* \geq x) \leq a_Y(x)\mathbb{P}(X \geq x), \qquad \forall x \geq x_0.
\end{equation}
\end{prop}

\begin{proof}
Using the formula \eqref{zerobiasdensity} for the density of $X^*$ and a change of measure, we obtain for any $t \geq x_0$,
\begin{equation*}
    f_{X^*}(t) = \frac{1}{\sigma^2} \mathbb{E}[X \mathbb{I}_{X > t}] = \frac{1}{\sigma^2}\mathbb{E}\left[Y \frac{f_X(Y)}{f_Y(Y)} \mathbb{I}_{Y > t} \right].
\end{equation*}
It follows that in the event $\{ Y \geq t\}$, $f_X(Y)/f_Y(Y) \leq f_X(t)/f_Y(t)$, then
\begin{equation*}
\frac{1}{\sigma^2}\mathbb{E}\left[Y \frac{f_X(Y)}{f_Y(Y)} \mathbb{I}_{Y > t} \right] \leq \frac{1}{\sigma^2}\mathbb{E}\left[Y \frac{f_X(t)}{f_Y(t)} \mathbb{I}_{Y > t} \right] =  \frac{f_X(t)}{f_Y(t)} \frac{1}{\sigma^2}\mathbb{E}\left[Y \mathbb{I}_{Y > t} \right] =  \frac{f_{Y^*}(t)}{f_{Y}(t)}f_{X}(t).
\end{equation*}
Therefore, for all $x \geq x_0$
$$\mathbb{P}(X^* \geq x) = \int_x^\infty f_{X^*}(t)dt \leq \int_x^\infty\frac{f_{Y^*}(t)}{f_{Y}(t)}f_X(t)dt \leq \int_x^\infty a_Y(x)f_X(t)dt  = a_Y(x)\mathbb{P}(X \geq x).$$
\end{proof}

In Proposition \ref{Prop:zerobiasKernel}, if we go back to taking $Y$ a Gaussian random variable we obtain $a_Y(x) = 1$ and once again, Strassen's theorem allows to get $X^* \le_{\sigma,\sigma} X$. One advantage of Proposition \ref{Prop:zerobiasKernel} is to be able to consider a sub-Gaussian behavior of $X$ only for large values of $x$. We use Proposition \ref{Prop:zerobiasKernel} in the following result.

\begin{cor}
Let $X$ be a random variable with zero mean, variance $\sigma^2$ finite and density $f_X = e^{-\varphi}$, where $\varphi$ satisfies that
\begin{enumerate}
    \item $\varphi'(x) \leq x/\sigma^2$, if $x \leq x_l < 0$,
    \item $\varphi'(x) \geq x/\sigma^2$, if $x \geq x_r > 0$.
\end{enumerate}
Then, $f_{X^*}(x) \leq f_X(x)$ for all $x \in (-\infty, x_l) \cup (x_r, \infty)$. In addition, if $\mathbb{E}X^3 =0$ and if for every $x \in (x_l,x_r)$ we have that $f_X(x) \leq f_{X^*}(x)$ then $X^*$ is smaller than $X$ in the convex order.
\end{cor}

\begin{proof}
Let $N$ be a gaussian random variable with mean zero, variance $\sigma^2$ and density $\phi$. The hypotesis $\varphi'(x) \geq x/\sigma^2$ implies that $f_X(x) / \phi(x)$ is decreasing for $x \geq x_r$. Then from Proposition \ref{Prop:zerobiasKernel} we obtain $f_{X^*}(x) \leq f_X(x)$ for all $x \geq x_r > 0$.

On the other hand, the density of $X^*$  given in Equation \eqref{zerobiasdensity} is for $x\leq x_l \leq 0$ equal to
\begin{equation*}
    f_{X^*}(x) = \frac{-1}{\sigma^2} \mathbb{E}[X \mathbb{I}_{X \leq x}] = \int_{-\infty}^x \frac{-t}{\sigma^2} f_X(t)dt = \int_{-\infty}^x \frac{-t}{\sigma^2}  \frac{f_X(t)}{\phi(t)} \phi(t)dt = \frac{-1}{\sigma^2}\mathbb{E}\left[N \frac{f_X(N)}{\phi(N)} \mathbb{I}_{N \leq x} \right].
\end{equation*}
 Using the hypothesis $\varphi'(x) \leq x/\sigma^2$, for $x \leq x_l$, we obtain that the function $f_X /\phi$ is increasing for $x \leq x_l$, then
$$\frac{-1}{\sigma^2}\mathbb{E}\left[N \frac{f_X(N)}{\phi(N)} \mathbb{I}_{N \leq x} \right] \leq  \frac{-1}{\sigma^2}\mathbb{E}\left[N \frac{f_X(x)}{\phi(x)} \mathbb{I}_{N \leq x} \right] =  \frac{f_X(x)}{\phi(x)} \frac{-1}{\sigma^2}\mathbb{E}\left[N \mathbb{I}_{N\leq x} \right] =  \frac{f_X(x)}{\phi(x)}\phi(x) = \phi(x),$$
that is $f_{X^*}(x) \leq f_X(x)$ for all $x \leq x_l$.

Finally by Definition \ref{defi_zerobias}$, \mathbb{E}X^3 =0$ implies that $\mathbb{E}X^* =0$. Also if for $x \in (x_l,x_r)$ follows that $f_X(x) \leq f_{X^*}(x)$, we obtain that the function $f_X - f_{X^*}$ has two sign changes and the sign sequence is $+,-,+$, which implies (see for example Theorem 3.A.44 from \cite{MR2265633}) that $X^*$ is smaller than $X$ in the convex order.

\end{proof}

\begin{remark}
Notice that if $\varphi'(x) \geq x/\sigma^2$, for $x \geq x_r \geq 0$, then integrating from $x_r$ to $x$, we obtain $\varphi(x)-\varphi(x_r) \geq (x^2-x_r^2)/(2 \sigma^2)$, which implies that $f_X(x) K(x_r) \leq f_Y(x)$, where $f_Y$ is the density of a Gaussian random variable with zero mean and variance $\sigma^2$, and $K(x_r) = f_Y(x_r)/f_X(x_r)$. Then,
$$\mathbb{P}(X \geq x) \leq K(x_r) \mathbb{P}(Y \geq x), \qquad \forall x \geq x_r.$$
Analogously, if $\varphi'(x) \leq x/\sigma^2$, if $x \leq x_l \leq 0$, we obtain $f_X(x) K(x_l) \leq f_Y(x)$ and
$$\mathbb{P}(X \leq x) \leq K(x_l) \mathbb{P}(Y \leq x), \qquad \forall x \leq x_l.$$
\end{remark}

\subsection{Concentration of sums of weakly-dependent random variables}

The subject of general functions of independent random variables is key for a enormous class of probabilistic models. It is well known that under mild conditions of the continuity of the function, one can show a concentration of measure phenomenon for the values of the function.
The case when the random variables are slightly dependent is less classical and somewhat lacks of general results. In this section, we develop new ways to show concentration when we even permit a little dependence between the random variables.
In particular, we are interested in the links between sub-Gaussianity of a multivariate function $f$ of random variables $X_1,\dots,X_n$ and zero-bias transforms of this function.
To keep track of the effect of each random variable on the rest of the random variables, we define a notion of directional zero bias transform.

\begin{defi}\label{DefTransfdireccion}
Let $X_1, X_2, \ldots, X_n$ be random variables with mean zero and finite variance  $\mathrm{Var}(X_i) = \sigma_i^2$. We define the directional zero bias distribution in the direction $i$,  as the random vector  $(X_1^{*(i)}, \ldots, X_n^{*(i)})$ that satisfies
\begin{equation}\label{zerobiasdirecion}
   \mathbb{E}[X_if(X_1, \ldots, X_n)] = \sigma_i^2 \mathbb{E}[\partial_i f(X_1^{*(i)}, \ldots, X_n^{*(i)})]
\end{equation}
where $f$ is a real function on $n$ variables such that the expectations exist.
\end{defi}
The first important remark is that no independence nor (un-)correlation of the random variables are necessary to be able to define the directional zero-bias transforms $X_j^{*(i)}$.
The following proposition states the basic facts of this directional zero-bias transform. Some results are direct consequences of Proposition \ref{prop_Zerobiasproperties}.

\begin{prop}\label{corPropZerobiasdire}
Let $X_1, X_2, \ldots, X_n$ be random variables with mean zero and finite variance  $\mathrm{Var}(X_i) = \sigma_i^2$. Then
\begin{enumerate}
    \item The random variable $X_i^{*(i)}$ has the same distribution as $X_i^*$ which is the classical zero-bias transform of Definition \ref{defi_zerobias}.
    \item If the random vector have a density, then the random variable $X_j^{*(i)}$ have a density given by
    \begin{equation}\label{densidadX_j*i}
        f_j^{*(i)}(y)= \int \frac{x_i^2 }{\sigma_i^2}f_{ij}(x_i,y)dx_i,
    \end{equation}
    where $f_{ij}$ is the joint density of $X_i,X_j$.
    \item If $X_i$ is independent of $X_j$, then $X_j^{*(i)} = X_j$
    \item $(aX_j)^{*(i)} = a X_j^{*(i)}$ for any constant $a \neq 0$.
\end{enumerate}
\end{prop}
\begin{proof}
\begin{enumerate}
    \item[\textit{1.}] If $f$ is a function that only depends of the entry $X_i$, it follows that
    $$\mathbb{E}[X_i f(X_i)] = \sigma_i^2 \mathbb{E}[f'(X_i^{*(i)})],$$
    but from the Definition \ref{defi_zerobias}, the first term of the above equation is equal to $\sigma^2\mathbb{E}[f'(X_i^*)]$,  therefore $X_i^{*(i)}$ and $X_i^*$ have the same distribution.
    \item[\textit{2.}] Choosing a function of the form $x_if(x_j)$, we have that $\partial_i (x_if(x_j)) =  f(x_j)$, then
    \begin{equation}\label{defZeroXj}
        \mathbb{E}[X_i^2f(X_j)] = \sigma_i^2 \mathbb{E}\left[f\left(X_j^{*(i)}\right)\right].
    \end{equation}
    From the last equation, it follows that the density of the random variable $X_j^{*(i)}$ is equal to
    \begin{equation}
        f_j^{*(i)}(y)= \int \frac{x_i^2 f_{ij}(x_i,y)}{\sigma_i^2}dx_i,
    \end{equation}
    where $f_{ij}$ is the joint density of $X_i,X_j$.
    \item[\textit{3.}] By independence and using the Equation \eqref{defZeroXj}, we obtain
    $$\mathbb{E}[f(X_j)] = \mathbb{E}\left[f\left(X_j^{*(i)}\right)\right],$$
    then $X_j$ and $X_j^{*(i)}$ have the same distribution.
    \item[\textit{4.}]The result follows using \eqref{defZeroXj} and the function $g(x) = f(ax)$.
\end{enumerate}
\end{proof}

We now state the main result of this section. It is stated for a random variable that is a sum of centered random variables but no independence is formally needed. It basically shows that if every random variable is sub-Gaussian in the sense that its zero bias transform can be upper bounded by the original random variable and if the cross directional zero bias transform is small then the sum of the variables inherits the sub-Gaussian behavior.
\begin{teo}
\label{teo:subgaussianfunctional}
Let $\Delta_1, \Delta_2, \ldots, \Delta_n$ be random variables such that $\mathbb{E}\Delta_i=0$ and $\mathrm{Var} \Delta_i = \sigma_i^2$. Suppose that
\begin{enumerate}
    \item for $\Delta_i$ and his zero bias transform $\Delta_i^*$,
    \begin{equation}\label{hip_subgausDelta}
        \Delta_i^* \stdompar{\sigma_i,k_{ii}} \Delta_i,
    \end{equation}
    \item for all $i \neq j$,
    \begin{equation}\label{hip_Deltaij}
        \Delta_j^{*(i)} \stdompar{\sigma_i,k_{ji}} \Delta_j.
    \end{equation}
\end{enumerate}
Then $W  = \sum_{i=1}^n \Delta_i $ has sub-gaussian right tail of constant
$$K^2 = \sum_{i=1}^n \sigma_i^{2-2n}  \prod_{j=1}^n k_{ji}^2.$$
\end{teo}

We give some comments on the hypothesis of Theorem \ref{teo:subgaussianfunctional}.
One can note that the two conditions \eqref{hip_subgausDelta} and \eqref{hip_Deltaij} could be unified into $\Delta_j^{*(i)} \stdompar{\sigma_i,k_{ji}} \Delta_j$ for every $i,j$. We separated the two since they have different interpretations.
Equation \eqref{hip_subgausDelta} is basically a sub-Gaussian hypothesis on the random variables $\Delta_i$ as seen in Theorem \ref{Teo:subGaus_equivalences}. In the examples below, we detail various cases where one can show easily the condition \eqref{hip_subgausDelta}.
For the interpretation of Equation \eqref{hip_Deltaij}, Definition \ref{DefTransfdireccion} gives directly that
$$\sigma_i^2 \mathbb{E}f(\Delta_j^{*(i)}) = \mathbb{E}_{ij} \Delta_i^2 f(\Delta_j),$$
then the hypothesis \eqref{hip_Deltaij} is equal to
$$\mathbb{E}_{ij} \left[\Delta_i^2f(\Delta_i) \right] = \sigma_i^2 \mathbb{E}f(\Delta_j^{*(i)}) \leq k_{ji}^2\mathbb{E}_{ij} f(\Delta_i),$$
which finally writes as
\begin{equation}
        \mathbb{E}_{ij} \left[\left( \Delta_i^2- k_{ji}^2\right)f(\Delta_j) \right] \leq 0,
\end{equation}
for all increasing function $f$ such that the expectations exists.
Note that this last condition holds if for all $i.j$ the random variables $\Delta_i^2$ and $\Delta_j$ are negatively associated. Indeed, by definition of negative association we have that for any pair of increasing functions $g$ and $h$,
\begin{equation*}
    \esp{g(\Delta_i^2)h(\Delta_j)}\le \esp{g(\Delta_i^2)}\esp{h(\Delta_j)}.
\end{equation*}
Since the functions $x\mapsto x-\sigma_i^2$ and $f$ are both increasing, we have directly the condition \eqref{hip_Deltaij} for $k_{ji}=\sigma_i$. Concentration of sum of random variables negatively associated is already known and the interested reader can take a look at \cite{barbour1992poisson} for example.
To finish the comments on Theorem \ref{teo:subgaussianfunctional}, we would like to give the following general interpretation for this result. \textit{A sum of sub-Gaussian random variables for which their cross Stein's transform are small (in the sense of \eqref{hip_Deltaij}) is itself sub-Gaussian}.

In order to prove Theorem \ref{teo:subgaussianfunctional}, we introduce a general form of Strassen's theorem \cite{lindvall1999strassen} that will be used in Lemma \ref{lem:strassen}.

\begin{teo}[Strassen generalized domination]\label{teo:Strassen}
    Let $(\Omega, \Sigma, \preceq)$ be a partially ordered set for which the set
    $$M = \{(x,y) : x \preceq y \}$$
is closed in the product topology of $\Omega^2$. Let $P,P'$ two probability measures on $(\Omega, \Sigma)$ such that $P$ is stochastically dominated by $P'$. Then, there exist a coupling $(X,X')$ of law $\Pi$ with marginals $X\sim P$ , $X' \sim P'$ and such that $X \preceq X'$, $\Pi$- almost surely. The law $\Pi$ of the coupling acts on the measurable space $(\Omega^2, \Sigma^2)$.
\end{teo}

The following consequence of Strassen's theorem allows us to deduce a almost sure control of the $\Delta^{(i)}_j$ by $\Delta_j$ for a specific coupling when we assume the weighted stochastic domination between $\Delta^{(i)}_j$ and $\Delta_j$.

\begin{lem}\label{lem:strassen}
Let $X$ and $Y$ be two random variables such that $X \stdom Y$. Then, for all increasing and positive function $f$ there exist a coupling $\Pi_f$ such that
 $$\Pi_f \left( \{\sigma^2 f(X+z) \leq k^2 f(Y+z), \quad \forall z \in \mathbb{R} \}\right) = 1.$$
\end{lem}

\begin{proof}

We introduce the following relation: Let $f$ be an increasing continuous and positive function and $\sigma^2, k^2$ positive constants such that $\sigma^2 \leq k^2$. We define $\preceq_f$ as
\begin{equation}\label{def:order}
    x \preceq_f y \qquad \text{if and only if} \qquad \sigma^2 f(x +z) \leq k^2 f(y +z), \quad \forall z \in \mathbb{R}.
\end{equation}
We denote $x \asymp_f y$ if $x \preceq_f y$ and $y \preceq_f x$. Notice that $\preceq_f$ is a partial order in $\mathbb{R}$. Indeed, $x \preceq_f x$ because $f$ is positive and $\sigma^2 \leq k^2$. Let $x \preceq_f y$ and $y \preceq_f w$. Suppose that $x \npreceq_f w$, then exist $\hat{z} \in \mathbb{R}$ such that $\sigma^2 f(x + \hat{z}) > k^2 f(w + \hat{z})$, but with $x \preceq_f y$, this implies that $f(w + \hat{z}) < f(y+\hat{z})$.
As $f$ is a increasing function, we obtain $w+ \hat{z} < y +\hat{z}$, that is $w < z.$ Notice that for all $z \in \mathbb{R}$,
\begin{align*}
    w < y &\Longrightarrow w + z < y+z  \Longrightarrow f(w+z) \leq f(y+z)  \Longrightarrow \sigma^2 f(w+z) \leq k^2 f(y+z),
\end{align*}
that is, $w \preceq_f y$. Then, we obtain that $y \asymp_f w$, which implies $x \preceq_f w$, which is a contradiction, because we are supposing that $x \npreceq_f w$. This proves transitivity.

Now notice that
\begin{align*}
    M &= \{ (x,y): x \preceq_f y\}\\
    &= \{(x,y): \sigma^2 f(x+z) \leq k^2 f(y+z) , \quad \forall z \in \mathbb{R}\}\\
    &=  \left\{ (x,y): f(z) \leq f^{-1}\left( \frac{k^2}{\sigma^2}f(y+z)-x \right) , \quad \forall z \in \mathbb{R}\right\}\\
    &=  \left\{ (x,y):  \bigcap_z f^{-1} \left\{\left(-\infty, f^{-1}\left( \frac{k^2}{\sigma^2}f(y+z)-x \right) \right] \right\} \right\},
\end{align*}
which by the continuity of $f$ is a closed set.  Now, the hypothesis $X \stdom Y$ implies that $X \preceq_f Y$, and by Theorem \ref{teo:Strassen} we obtain that exist a coupling of law $\Pi_f$ such that $X \preceq_f Y$, $\Pi_f$- almost surely.
\end{proof}

\begin{proof}[Proof of Theorem \ref{teo:subgaussianfunctional}]
Let $f$ be a absolutely continuous increasing and convex function with continuous derivative. Using the Definition \ref{DefTransfdireccion} with the function $g(x_1, \ldots, x_n) = f \left(\sum_{i=1}^n x_i \right)$,
\begin{align*}
    \mathbb{E}Wf(W) &= \sum_{i=1}^n \mathbb{E} \Delta_i f(W) \\
    &= \sum_{i=1}^n \mathbb{E}\Delta_i g(\Delta_1, \Delta_2, \ldots, \Delta_n)\\
    &= \sum_{i=1}^n \sigma_i^2 \mathbb{E} \partial_i g(\Delta_1^{*(i)}, \ldots, \Delta_n^{*(i)}).
\end{align*}
From Corollary \ref{corPropZerobiasdire}, it follows that $\Delta_i^{*(i)} = \Delta_i^*$ for any $i = 1, \ldots, n$. First, we see the case $i=1$. The hypothesis $\Delta_1^* \stdompar{\sigma_1,k_{11}} \Delta_1$ with the Lemma \ref{lem:strassen}, implies that for the function $\partial_1 g$ which is increasing continuous and positive, $\Delta_1^* \preceq_{\partial_1 g} \Delta_1$ almost surely, then
\begin{align*}
    \mathbb{E} \partial_1 g(\Delta_1^{*(1)}, \Delta_2^{*(1)}, \ldots, \Delta_n^{*(1)})
    &\leq  \frac{k_{11}^2}{\sigma_1^2} \mathbb{E}\partial_1 g(\Delta_1,\Delta_2^{*(1)}, \ldots, \Delta_n^{*(1)}).
\end{align*}
Now, using the hypothesis $\Delta_j^{*(1)} \stdompar{\sigma_1,k_{j1}} \Delta_j$ implies that for any $j\neq 1$, with the function $\partial_1 g$ and the Lemma \ref{lem:strassen}, that $\Delta_j^{*(1)} \preceq_{\partial_1 g} \Delta_j$ almost surely, then
\begin{align*}
    \mathbb{E} \partial_1 g(\Delta_1, \Delta_2^{*(1)}, \Delta_3^{*(1)}, \ldots, \Delta_n^{*(1)}) &\leq  \frac{k_{21}^2}{\sigma_1^2}\mathbb{E}\partial_1 g(\Delta_1, \Delta_2, \Delta_3^{*(1)},\ldots, \Delta_n^{*(1)}),
\end{align*}
Then, with the same argument and noticing that in each step we use a different coupling, we obtain
$$\mathbb{E} \partial_1 g(\Delta_1^{*(1)}, \ldots, \Delta_n^{*(1)}) \leq \left(\prod_{j=1}^n \frac{k_{j1}^2}{\sigma_1^2} \right) \mathbb{E} \partial_1 g(\Delta_1, \ldots, \Delta_n).$$
With a completely analogous reasoning, we obtain for any $i$,
$$\mathbb{E} \partial_i g(\Delta_1^{*(i)}, \ldots, \Delta_n^{*(i)}) \leq \left(\prod_{j=1}^n \frac{k_{ji}^2}{\sigma_i^2} \right) \mathbb{E} \partial_i g(\Delta_1, \ldots, \Delta_n).$$
On the other hand,  $g(x_1, \ldots, x_n) = f(\sum_{i=1}^n x_i)$, then for any $i$, $\partial_i g(x_1, \ldots, x_n) = f'(\sum_{i=1}^n x_i),$ that is
$$\mathbb{E} \partial_i g(\Delta_1^{*(1)}, \ldots, \Delta_n^{*(1)}) \leq   \left(\prod_{j=1}^n \frac{k_{ji}^2}{\sigma_i^2} \right)\mathbb{E} f'(\Delta_1+ \ldots + \Delta_n) = \left(\prod_{j=1}^n \frac{k_{ji}^2}{\sigma_i^2} \right) \mathbb{E} f'(W).$$ Therefore, we get that
\begin{align*}
    \sigma_W^2 \mathbb{E}f'(W^*) &= \mathbb{E}Wf(W) \\
    &=\sum_{i=1}^n \sigma_i^2 \mathbb{E} \partial_i g(\Delta_1^{*(i)}, \ldots, \Delta_n^{*(i)}) \\
    &\leq \sum_{i=1}^n \sigma_i^2  \left(\prod_{j=1}^n \frac{k_{ji}^2}{\sigma_i^2} \right)\mathbb{E}f'(W) \\
    &=  \mathbb{E}f'(W) \sum_{i=1}^n \sigma_i^{2-2n} \prod_{j=1}^n k_{ji}^2\\
    &=  K^2 \mathbb{E}f'(W),
\end{align*}
for any increasing continuous and positive function $f'$. Now, any increasing and positive function $f$ can be approximated by increasing continuous and positive functions, then we obtain from Theorem \ref{Teo:subGaus_equivalences} that $W \in \mathcal{G}_+(K^2)$.
\end{proof}

\subsection{Examples of applications of Theorem \ref{teo:subgaussianfunctional}}
In this section, we present four examples of increasing complexity that show the usefulness of Theorem \ref{teo:subgaussianfunctional}. The first three examples deal with sums of light-tailed random variables for three different dependence structure. The last example allows us to tackle the case of general functions of weakly dependent random variables.

\paragraph{Example 1. (Dependency Neighborhoods)} We say that a collection of random variables $\{Y_i\}_{1\le i \le n}$ has dependency neighborhoods $N_i \subset [n]$ with $i=1, \ldots, n$ if for every $1\le i\le n$ we have that $i \in N_i$ and $Y_i$ is independent of $\{Y_j\}_{j \notin N_i}$. In particular, this property implies that $Y_j^{(i)} = Y_j$ for any $j \notin N_i$.

Suppose, now, that the random variables $\{\Delta_i \}_{i=1, \ldots, n}$ have dependency neighborhoods $\{N_i\}_i$. Then, we can apply Theorem \ref{teo:subgaussianfunctional} with the choice $k_{ji}^2=\sigma_i^2$ in  \eqref{hip_Deltaij} for every $j \notin N_i $. Therefore
$$K^2 = \sum_{i=1}^n \sigma_i^{2-2n} \prod_{j=1}^n k_{ji}^2 = \sum_{i=1}^n \sigma_i^{2-2n} \prod_{j \in N_i} k_{ji}^2 \prod_{j \notin N_i}\sigma_i^2   =\sum_{i=1}^n \sigma_i^{2-2|N_i|}  \prod_{j \in N_i} k_{ji}^2,$$
and $W$ has sub-Gaussian right tails of constant $K^2$. In the special case where all $\{\Delta_i \}_{i=1, \ldots, n}$ are independents, we obtain that $k_{ji}^2 = \sigma_j^2$ for all $i \neq j$ and $|N_i| = 1$ which gives
$$K^2=\sum_{i=1}^n \sigma_i^{2-2|N_i|}  \prod_{j \in N_i} k_{ji}^2  = \sum_{i=1}^n  k_{ii}^2.$$
As seen in the definition of the weighted stochastic domination, the constants $k_{ji}$ are forced to be greater than the variance $\sigma_j$. In consequence, the most constants $k_{ji}$ being equal to $\sigma_j$, the smaller the constant $K$ is.
This case is typically adapted to random graphs where the dependence in the random variables is fairly local. In most of the classical cases (Erdós-Reyni random graphs for example), the dependence neighbor is microscopic ($|N_i|\ll n$). In that case, the constant $K^2$ of full independence case is of the good order.

\paragraph{Example 2. (Linear combination of orthogonal random variables)} Let $\{Y_i \}_{i = 1, \ldots n}$ be a collection of orthogonal random variables with zero mean and unit variance.
Assume that $Y_i^* \stdompar{1,k_{ii}} Y_i$ for all $i=1, \ldots n$ and $Y_j^{*(i)} \stdompar{1, k_{ji}} Y_j$ for all $j \neq i$. For any vector $a =( a_1, \ldots, a_n)$ with non negative entries, define $Y_a = \sum a_i Y_i$ and $\sigma^2 = \mathrm{Var}(Y_a) = \sum a_i ^2 = ||a||^2$.
Finally, define
$$Z_a = \frac{1}{\sigma} \sum_{i =1}^n a_i Y_i =: \sum_{i =1}^n \Delta_i.$$
The variances of the variables $\Delta_i$ are given by $\sigma_i^2 = a_i^2 / \sigma^2.$ The two hypothesis on the class $\{Y_i\}_i$ imply that $\Delta_i^{*} \stdompar{\sigma_i, \sigma_i k_{ii}} \Delta_i$ for all $i= 1, \ldots, n$ and $\Delta_j^{*(i)} \stdompar{\sigma_i, \sigma_i k_{ji}} \Delta_j$ for all $j \neq i$ respectively. Indeed, for any increasing and positive function $f$, consider the function $g(x) = f(\sigma_i x)$ and apply the Definition \ref{def:stdom}. Then, we obtain
$$K^2 = \sum_{i=1}^n \sigma_i^{2-2n} \prod_{j = 1}^n \sigma_i^2 k_{ji}^2 = \sum_{i=1}^n \sigma_i^{2} \prod_{j = 1}^n k_{ji}^2  = \sum_{i=1}^n \frac{a_i^2}{\sigma^2} \prod_{j = 1}^n  k_{ji}^2 = \frac{1}{||a||^2}\sum_{i =1}^n a_i^2 \prod_{j=1}^n k_{ji}^2.$$
Once again, one can see that when we assume that the random variables $Y_i$ are independent, all the constants $k_{ji}=1$ for $i\neq j$. We are reduce to consider a sum of independent sub-Gaussian random variables. Hence, Theorem \ref{teo:subgaussianfunctional} shows that in that case $Z_a$ is sub-Gaussian $\mathcal{G}(K^2)$ with $K^2=\frac{1}{||a||^2}\sum_{i =1}^n a_i^2 k_{ii}^2$.
By analogy with Gaussian vectors, we see that the sub-Gaussian property is stable by linear combinations and the resulting sub-Gaussian constant behaves exactly in the same way than the variance for Gaussian random vectors.

\paragraph{Example 3. (Bounded random variables)} In this example, we consider the variable $W=\sum_i \Delta_i$ of Theorem \ref{teo:subgaussianfunctional} and assume that the random variables $\{ \Delta_i\}_{i=1, \ldots, n}$ have compact support, namely $\mathrm{Supp} \Delta_i = [a_i, b_i]$, with $a_i < 0 < b_i$, and we also assume that the random variables $\Delta_i$ have unimodal density $f_{\Delta_i}$.
Then, from direct calculation using \eqref{zerobiasdensity}, we obtain $f_{\Delta_i^*}(x) \leq \max(a_i^2, b_i^2)/2 f_{\Delta_i}(x)$, which implies that
$$\sigma_i^2 \mathbb{E}f(\Delta_i^*) \leq k_{ii}^2 \mathbb{E}f(\Delta_i),$$
for any increasing and positive function $f$ such that the expectations exist and with the constants $k_{ii}^2 = \max(a_i^2, b_i^2)/2$. Additionally, by \eqref{defZeroXj} and the fact that $\Delta_i$ is of compact support, we see that
$$\sigma_i^2 \mathbb{E}f'(\Delta_j^{*(i)}) = \mathbb{E}_{ij} \Delta_i^2 f'(\Delta_j) \leq \max(a_i^2, b_i^2) \mathbb{E}f'(\Delta_j).$$
In terms of weighted stochastic domination, the later is equivalent to $\Delta_{j}^{*(i)} \stdompar{\sigma_i, k_{ji}} \Delta_j$ with $k_{ji}^2 = \max(a_i^2, b_i^2)$. Then, by Theorem \ref{teo:subgaussianfunctional}, $W = \sum_i \Delta_i$ has sub-Gaussian tails of constant
$$K^2 = \sum_{i =1 }^n \sigma_i^{2-2n} \frac{1}{2} \max(a_i^2, b_i^2) \prod_{j\neq i}\max(a_i^2, b_i^2)  =  \sum_{i =1 }^n \sigma_i^{2-2n} \frac{1}{2} \max(a_i^{2n}, b_i^{2n}).$$
If the random variables $\{ \Delta_i\}_{i=1}^n$ are independent, we can take $k_{ji}^2 = \sigma_i^2$, and obtain
$$K^2 = \frac{1}{2} \sum_{i=1}^n \max(a_i^2, b_i^2) \leq \frac{1}{2} \sum_{i=1}^n (b_i - a_i)^2.$$
Note that this last constant is equal to the constant of sub-Gaussian constant given by the classical Hoeffding's inequality for sums of bounded independent variables. Once again, we see that Theorem \ref{teo:subgaussianfunctional} generalizes a classical result on sub-Gaussian concentration.

\paragraph{Example 4. (Function with bounded differences)} In this example, we are interested to see the consequences of our main theorem for functions of independent random variables. It is known in the concentration of measure community that a function $f$ that depends regularly on its entries and if no random variable have a ``macroscopic" influence on the value of the function, then $f$ concentrates around its mean value $\esp{f}$. In this example, we obtain directly McDiarmid inequality on bounded difference functions (see \cite{boucheron2013concentration} for a compehensive proof) by direct use of previous example. We say that a function $f : \mathcal{X}^n \to \mathbb{R} $  has the bounded differences property if for some nonnegative constants  $c_1, \ldots, c_n$,
\begin{equation}\label{def:boundeddifferencesfunction}
    \sup_{x_1, \ldots, x_n \\~ \\ x_i\in \mathcal{X}} |f(x_1, \ldots, x_n)- f(x_1, \ldots, x_{i-1}, x_i', x_{i+1}, \ldots, x_n) | \leq c_i,
\end{equation}
for $1 \leq i \leq n$.

Let $X_1, \ldots, X_n$ be independent random variables in some space  $\mathcal{X}$ and let $f$ a function with bounded differences and consider the random variable $Z = f(X_1, \ldots, X_n)$. We can write $Z - \mathbb{E}Z$ as a sum of martingale differences for the Doob filtration as
\begin{equation}\label{martingaledifferences}
Z- \mathbb{E}[Z] = \sum_{i=1}^n (Z_i-Z_{i-1}),
\end{equation}
where
\begin{equation}
    Z_i = \mathbb{E}[Z | X_1, \ldots, X_i].
\end{equation}
The bounded differences property of $f$ implies that the support of the random variables $\Delta_i=Z_i - Z_{i-1}$ is a subset of $[a_i, b_i]$ for some constants $a_i < 0  < b_i$ with $b_i- a_i \leq c_i$.
The independence of the random variables $X_1,\dots,X_n$ implies that the random variables $\Delta_i$ are uncorrelated, then, we can use the previous example to obtain that $Z$ has sub-Gaussian tails. Furthermore, if the support of the random variables is symmetric, meaning that $-a_i = b_i = c_i/2$, we obtain that $\sigma_i^2 = \mathrm{Var}(Z_i-Z_{i-1}) \leq (b_i-a_i)^2/4 = c_i^2/4$, and
$$K^2 = \sum_{i =1 }^n \sigma_i^{2-2n} \frac{1}{2} \max(a_i^{2n}, b_i^{2n}) = \sum_{i=1}^n \left( \frac{c_i}{2}\right)^{2-2n} \frac{1}{2} \left( \frac{c_i}{2}\right)^{2n} = \frac{1}{8} \sum_{i=1}^n c_i^2.$$
Note that the constant $K^2$ is precisely the constant of sub-Gaussianity present in \cite[Theorem 6.2]{boucheron2013concentration}.

\subsection{Hoeffding's statistics}
In this section, we present a concrete example where the random variables that come into play are naturally slightly dependent. The object of interest is the Hoeffding statistic given by
\begin{equation}\label{HoeffdingStatistics}
    Y = \sum_{i=1}^n a_{i\pi(i)}
\end{equation}
where $\pi$ be a random uniform permutation in the symmetric group $S_n$ and $A = (a_{ij})_{1 \leq i,j \leq n}$ is an $n \times n$ matrix with real entries. This random variable is of great interest in permutation test (see for example \cite{chen2011normal}), or in a kind of limit theorems, called Hoeffding combinatorial central limit theorems (see \cite{MR44058}).
Concentration bounds for $Y$ are given by \cite{MR2288072,goldstein2014concentration}, where they obtain sub-Gamma right tails of type

\begin{equation}\label{subgamaChatt_bound}
    \mathbb{P}\left( |Y - \mathbb{E}Y | \geq t\right) \leq 2 \exp\left( \frac{-t^2}{4\mathbb{E}Y +2t} \right),
\end{equation}
or
\begin{equation}\label{subgamaGolds_bound}
    \mathbb{P}\left( |Y - \mathbb{E}Y | \geq t\right) \leq 2 \exp\left( \frac{-t^2}{2\mathrm{Var}(Y) +16t} \right).
\end{equation}
We see in the sequel that our approach allows us to obtain sub-Gaussian tails for this statistics.
By a direct calculation we obtain the first and second moments of the random variable $a_{i\pi(i)}$,
$$\mathbb{E}a_{i\pi(i)} = \frac{1}{n} \sum_{j=1}^n a_{ij}, \qquad \text{and}\qquad \mathbb{E}a_{i\pi(i)}^2 = \frac{1}{n} \sum_{j=1}^n a_{ij}^2.$$
Define
\begin{equation}
    \hat{Y} = Y - \mathbb{E}Y = \sum_{i=1}^n \left(\frac{1}{n} \sum_{j=1}^n(a_{i\pi(i)} - a_{ij}) \right) = \sum_{i=1}^n \Delta_i.
\end{equation}
The random variables $\Delta_i$ only depend on the entries of the row $i$ of the matrix $A$, then the support of the random variable $\Delta_i$ lies in the interval $[-c_i, c_i]$, where $c_i = \max_{j}a_{ij}- \min_{j}a_{ij}$. Then, using Example 3, we obtain that $\hat{Y} \in \mathcal{G}_{+}(K^2)$, with
$$K^2  \leq \frac{1}{2} \sum_{i=1}^n c_i^2 = \frac{1}{2} \sum_{i=1}^n \left(\max_{j}a_{ij}- \min_{j}a_{ij} \right)^2, $$
which implies that for all $t \geq 0$,

\begin{equation}\label{hoefdingTail}
    \mathbb{P}(\hat{Y} \geq  t) \leq \exp \left( -\frac{t^2}{ \sum_{i=1}^nc_i^2}\right).
\end{equation}
It is also posible to extend this concentration result to functions of the random variables $a_{i\pi(i)}$. Indeed, for $f$ a $L$-Lipschitz function, consider the random variable
\begin{equation*}\label{hoeffdingLipschitz}
Z = f(a_{1\pi(1)}, \ldots, a_{n\pi(n)}).
\end{equation*}
Since the function $f$ only depends of the entries of the matrix $A$, we obtain that $f$ is a function with bounded differences with constants $Lc_1, \ldots , L c_n$. Then $Z- \mathbb{E}Z$ has sub-Gaussian tails with constant $$K_L^2 \leq  \frac{L^2}{2} \sum_{i=1}^n \left(\max_{j}a_{ij}- \min_{j}a_{ij} \right)^2,$$
which implies that for all
 $t \geq 0$,
\begin{equation}\label{hoefdingTailZ}
    \mathbb{P}(Z - \mathbb{E}Z \geq  t) \leq \exp \left( -\frac{t^2}{ L^2\sum_{i=1}^nc_i^2}\right).
\end{equation}

\section{Sub Poisson and Size bias control}

\subsection{Size bias Coupling for real valued random variables}

In this section we expose a similar result to Theorem \ref{teo:subgaussianfunctional} in the case of the size bias transform. In \cite{Arratia2015Boundedsizebias}, the authors show that if a coupling $(X,X^s)$ such that $X^s\le X+c$ almost surely exists then the random variable $X$ has a sub-Gamma (in fact sub-Poisson) right tail.
The results presented below deal with the converse: What conditions are necessary to obtain a coupling between the two non-negative random variables $X^s$ and $X$ and such that $X^s\le X+c$ almost surely?
For any positive random variable $X$ with finite mean $\mu$, the size bias  transform $X^s$ exists as seen in Definition \ref{def:sizebias}. When there exists a positive constant $c$ such that $X^s \leq X+c$ almost surely, the result \cite[Lemma 2.1]{Arratia2015Boundedsizebias} shows that for all $t\geq 0$,

\begin{equation}
\label{eq:poissonlossmemory}
    \mathbb{P}(X \geq t) \leq \frac{\mu}{t} \mathbb{P}(X \geq t-c).
\end{equation}
This allows to obtain a concentration bounds of type \cite[Theorem 1.3]{Arratia2015Boundedsizebias}\begin{equation*}
    \mathbb{P}(X \geq t) \leq \exp\left( \frac{t-\mu}{c} - \frac{t}{c} \log\left( \frac{t}{\mu}\right)\right).
\end{equation*}
The authors refer to it as Gamma function bound since it decays at least as fast as the reciprocal of a Gamma function. We are interested in the converse problem, that is, to understand when we can obtain a bounded coupling.

The next theorem give conditions for the existence of a bounded coupling when we have an inequality between the density (or the probability mass function) of the size bias transform and a right-shift of the density of the random variable. See Figure \ref{img} for a visual interpretation. This condition is a density version of \eqref{eq:poissonlossmemory}.

\begin{teo}\label{teo:boundsizebias}
Let $X$ be a positive random variable with finite mean $\mu$. Suppose that exist positive constants $c$ and $t_0$ such that
\begin{equation}\label{sizebiascondition}
    \frac{x}{\mu}f_X(x) \leq f_X(x-c), \qquad \forall x \geq t_0,
\end{equation}
where $f_X$ is the probability mass function if $X$ is discrete or the density function if $X$ is absolutely continuous. Then, for $X^s$ having the size bias distribution of $X$,
\begin{equation}\label{eqsizebiasprob}
    \mathbb{P}(X^s \geq t) \leq \mathbb{P}(X+c \geq t), \qquad \forall t \geq t_0.
\end{equation}

\end{teo}

\begin{proof}
For Definition \ref{sizebiasdensity}, we have that the density of $f_{X^s} = x f_X / \mu$, then the inequality \eqref{sizebiascondition} is  $f_{X^s}(x) \leq f_X(x-c)$ for all $x \geq t_0$. Therefore, for any $t \geq t_0,$
$$\mathbb{P}(X^s \geq  t) = \int \mathbb{I}_{\{x \geq t\}}f_{X^s}(x)dx  \leq  \int \mathbb{I}_{\{x \geq t\}} f_{X}(x-c)dx = \mathbb{P}(X \geq t-c),$$
that is, $ \mathbb{P}(X^s \geq t) \leq \mathbb{P}(X+c \geq t)$, for all $t \geq t_0.$
\end{proof}

Equation \eqref{eq:poissonlossmemory} and Theorem \ref{teo:boundsizebias} together show the equivalence between sub-Gamma concentration and the existence of a coupling such that $X^s\le X+c$.
Given that we are in the case of positive random variables, by a direct calculation we obtain that if $x \leq \mu$, then  $f_{X^s}(x) \leq f(x)$ and if $x \geq \mu$, then $f_{X^s}(x) \geq f(x)$, so the constants $c$ and $t_0$ in the previous theorem are the ones corresponding to the context where a shift of the original density by $c$ upper bounds the density of $X^s$. This way the condition \eqref{sizebiascondition} is satisfied, as we show in the Figure \ref{img}.

\begin{figure}[h] \label{img}
\centering
\includegraphics[width=12
cm]{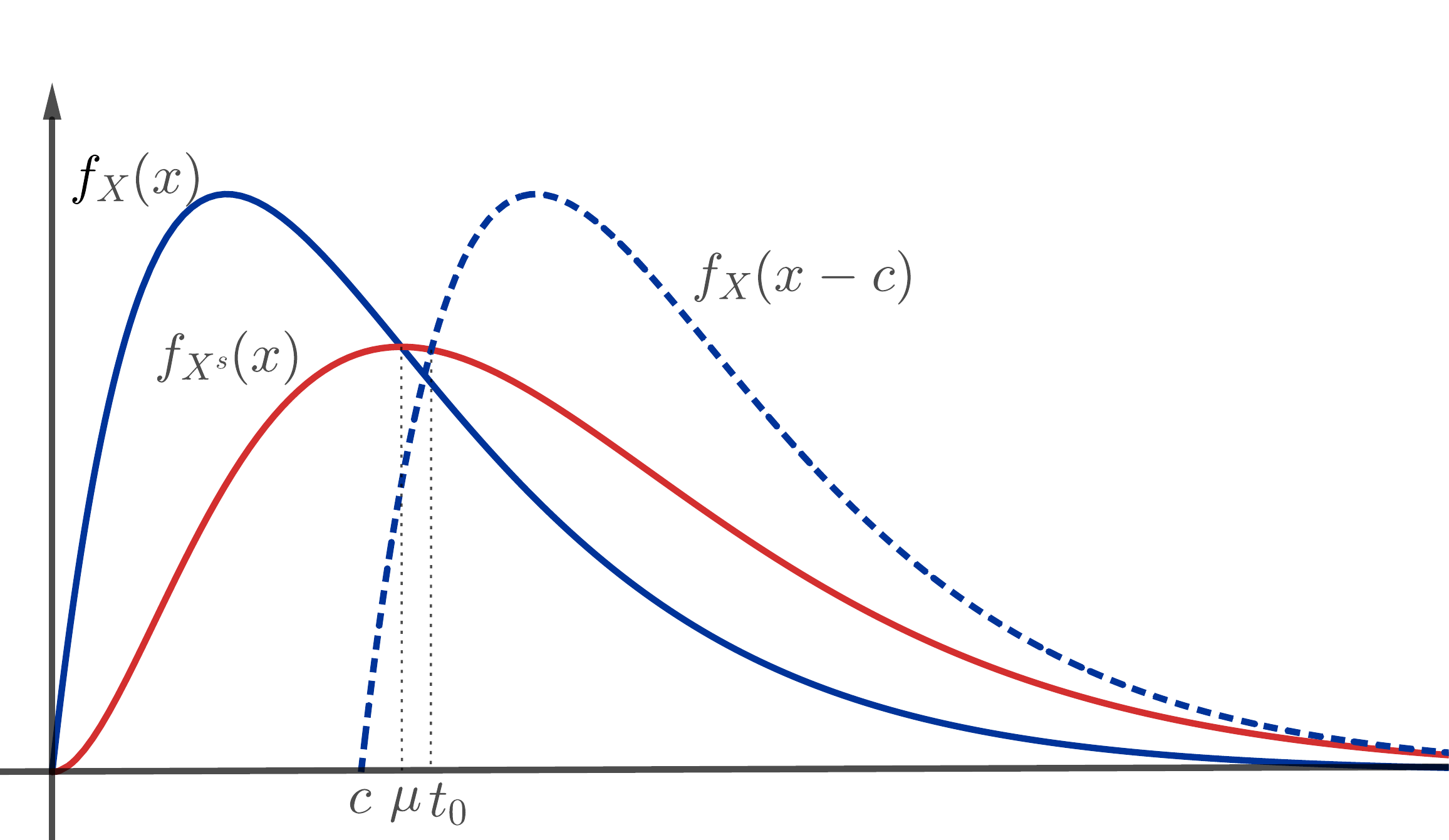}
\caption{Density of $f_X$ and $f_{X^s}$ with the shift by $c$ as in Theorem \ref{teo:boundsizebias}.}
\end{figure}

\subsection{Sum of independent random variables}

One of the most important properties of the size bias transform is that it has an easy explicit construction for sums of random variables. The case of sums of independent random variables is almost trivial since the construction rely only on randomly replacing a random variable for his size bias version with probability proportional to its mean, as shows the next result  \cite[Corollary 2.1]{chen2011normal}.

\begin{cor}
Let $Y = \sum_{i = 1}^n X_i$, where $X_1,\dots, X_n$ are independent, nonnegative random variables with means $\mathbb{E} X_i = \mu_i, i = 1, \ldots, n$. Let $I$ be a random index with distribution given by
$$\mathbb{P}(I=i) = \frac{\mu_i}{\sum_{i=1}^n \mu_i},$$
independent of all other variables. Then, upon replacing the summand $X_I$ selected by $I$ with a variable $X_I^s$ having its size biased distribution, independent of $X_j$ for $j \neq I$, we obtain
$$Y^I = Y - X_I + X_I^s,$$
a variable having the $Y$-size bias distribution.
\end{cor}
Note that there exists a general result that gives a method to obtain the size bias transform for a sum of dependent random variables. In this case, to replace a random variable by his size bias transform, it is necessary to adjust the remaining random variables conditionally to this new value. For further details see \cite[Proposition 2.2]{chen2011normal}.

The following result shows the behavior of the sum of independent random variables that have sub-Gamma concentration in the sense of Equation \eqref{eq:poissonlossmemory}. As in Theorem \ref{teo:subgaussianfunctional}, the sum inherits of the almost sure control of its size bias transform by a constant plus the original random variable.
\begin{teo}
  \label{teo:subGamma}
Let $X_1, \ldots, X_n$ be independent, nonnegative random variables with means $\mathbb{E} X_i = \mu_i, i = 1, \ldots, n$, such that the size bias transform of each $X_i$ satisfies that $X_i^s \leq X_i +c_i$, for some constants $c_i$. Then, for $Y = \sum_{i = 1}^n X_i$ we have that
$$Y^s \leq Y + \max c_i.$$
\end{teo}

It is interesting to see that the behavior of the constant $C$ in the result $Y^s\le Y+C$ is fundamentally different from Theorem \ref{teo:subgaussianfunctional}. Indeed, here the resulting constant is the maximum of all the constants $c_i$.

\begin{proof}
Let $Y^{(i)}= Y + X_i^s - X_i$. Then, for each $i=1, \ldots, n$,
\begin{align*}
    \mathbb{P}(Y^{(i)} \leq Y +  \max c_i ) &= \mathbb{P}(Y + X_i^s - X_i\leq Y +  \max c_i ) \\
    &= \mathbb{E}[\mathbb{E}[\mathbb{I}_{Y + X_i^s - X_i\leq Y +  \max c_i}| \sum_{j \neq i} X_j] \\
    &= \mathbb{E}[\mathbb{E}[\mathbb{I}_{X_i^s \leq X_i +  \max c_i}| \sum_{j \neq i} X_j] \\
    &= \mathbb{E}[\mathbb{I}_{X_i^s \leq X_i +  \max c_i}] \\
    &= \mathbb{P}(X_i^s \leq X_i +  \max c_i) \\
    &\geq \mathbb{P}(X_i^s \leq X_i +  c_i) =1
\end{align*}
Then, conditioning on $I$, we obtain
$$\mathbb{P}(Y^s \leq Y + \max c_i) = \sum_{i=1}^n \mathbb{P}(Y^{(i)} \leq Y +  \max c_i ) \mathbb{P}(I=i) \geq \sum_{i=1}^n \mathbb{P}(I=i) =1 .$$
\end{proof}

For a sanity check, one can note that for a Poisson random variable $X$ of parameter $\lambda$, $X^s=X+1$ and so if we have a class $\{X_i\}_i$ of Poisson random variables of parameters $\{\lambda_i\}_i$, we have trivially $c_i=1, \forall i$. But, since the sum of independent Poisson random variables is also Poisson, we have that $Y^s=Y+1$ which is coherent with the constant resulting of Theorem \ref{teo:subGamma} where $1=\max_i c_i$. Hence, Theorem \ref{teo:subGamma} can be considered optimal in that sense.

\bibliographystyle{plain}
\bibliography{references}

\end{document}